\documentclass{amsart}

 \usepackage{tabularx}
 \usepackage{booktabs} 
 \usepackage{array}    
\usepackage{amsmath,amsthm,amssymb,amscd,enumerate,mathtools,enumitem, tikz-cd}
\usepackage{amsfonts}
\usepackage{rotating}
\usepackage{euscript}
\usepackage{pst-node}
\usepackage{epsfig,verbatim}
\usepackage{stackengine,scalerel}
\stackMath

\def\be{\begin{equation}}
\def\ee{\end{equation}}

\def\P{{\mathbb P}}

\def\Q{{\mathbb Q}}

\def\phi{{\varphi}}

\def\deg{{\rm deg\,}}

\def\bp{\begin{proposition}}
\def\ep{\end{proposition}}

\def\bt{\begin{theorem}}
\def\et{\end{theorem}}
\def\br{\begin{remark}}
\def\er{\end{remark}}
\def\be{\begin{equation}}
\def\bee{\begin{equation*}}
\def\l{\label}

\def\ee{\end{equation}}
\def\eee{\end{equation*}}
\def\bl{\begin{lemma}}
\def\el{\end{lemma}}
\def\bc{\begin{corollary}}
\def\ec{\end{corollary}}
\def\pr{\noindent{\it Proof. }}

\def\bd{\begin{definition}}
\def\ed{\end{definition}}

\def\tilde{\widetilde}
\def\h{\widehat}

\def\hat{\widehat}

\newtheorem{theorem}{Theorem}[section]
\newtheorem{lemma}[theorem]{Lemma}
\newtheorem{definition}[theorem]{Definition}
\newtheorem{corollary}[theorem]{Corollary}
\newtheorem{proposition}[theorem]{Proposition}
\newtheorem{problem}[theorem]{Problem}

\theoremstyle{definition}

\theoremstyle{definition}
\newtheorem{remark}[theorem]{Remark}
\def\bpr{\begin{problem}}
\def\epr{\end{problem}}

\mathtoolsset{showonlyrefs}

\begin{document}

\title{Complex Analysis and Existence Problems for plane  Graphs}
\date{}
\author[F. Pakovich]{Fedor Pakovich}
%\thanks{
%This research was supported by ISF Grant  No. 1092/22}
%\address{Department of Mathematics, Ben Gurion University of the Negev, Israel}
%\email{
%pakovich@math.bgu.ac.il}

%\keywords{}
%\subjclass[2010]{}

\begin{abstract}
We show that a variety of known and new results concerning connected plane graphs whose vertex and face degrees satisfy prescribed uniformity conditions with at most two exceptions can be deduced from recent results on the Hurwitz existence problem regarding the realizability of branch patterns of rational functions. Our method also yields a description of the Belyi functions corresponding to such graphs.
\end{abstract}

\maketitle

\section{Introduction}

The problem of the existence of polyhedra and plane graphs possessing specific regularity properties with respect to their vertex and face degrees has a long and rich history. The origin of this line of research dates back to the classical classification of Platonic solids, which are polyhedra that are both vertex- and face-regular. By virtue of the remarkable correspondence between plane graphs and Belyi functions provided by the theory of \emph{dessins d'enfants}, the existence problem for a plane graph with prescribed vertex and face degrees is equivalent to the existence problem for a rational function with three critical values and a specified ramification pattern over them. The latter is a special case of the classical Hurwitz existence problem concerning the realizability of ramification patterns of coverings between surfaces (see \cite{pepe} for an introduction to the subject).

The correspondence between plane graphs and functions has long been successfully utilized, 
though predominantly in a unidirectional manner. Namely, constructing a graph with 
certain properties (or proving its non-existence) typically yields a corresponding 
result for rational functions (see \cite{mar, lz, ko, hu, pz, pp} for just a few examples). 
In the present paper, we provide several examples demonstrating how this correspondence 
operates in the opposite direction. Specifically, we show how the non-existence of 
certain rational functions can be employed to establish the non-existence of plane graphs 
with specific properties. Moreover, our method imposes constraints not only on the 
vertex and face degrees of the graphs under consideration, but also on their geometry. 
The majority of the results in this paper build upon the recent work \cite{pak}. 
However, to make the exposition self-contained, we provide alternative proofs for 
the required results from \cite{pak}, which we believe are of independent interest.

In brief, this paper explores plane graphs that are, in a certain sense, close to the famous Platonic graphs, which are defined as being both vertex-regular and face-regular. Our results demonstrate that when such a graph exists, its Belyi function is algebraically related to the Belyi function of the corresponding Platonic graph. In turn, these relations provide a specific analytic intuition that is not immediately apparent from the purely combinatorial picture.

Throughout this paper,  the term 
``graph'' exclusively refers to a connected plane graph that may contain multiple 
edges and loops. By ``Platonic graphs,'' we mean graphs associated with 
Platonic solids completed by cycles. Note, however, that some of the results 
cited below were originally formulated under more restrictive assumptions. 
Since we provide independent proofs for these results in our setting anyway, 
we do not dwell on the precise definitions of graphs or triangulations used in the 
cited literature, referring the reader to the original works.

The definition of ``closeness'' to a Platonic graph has several versions, but it always implies certain conditions on vertex and face degrees formulated in terms of divisibility or equality, with a possible small number of exceptions. One class of graphs falling under this description is the class of spherical triangulations, that is, graphs whose face degrees are \emph{equal} to $3$, with the property that the degrees of all vertices are \emph{divisible} by a given integer $k$. We start by mentioning several results about such graphs or their duals. Note that by Euler's formula, the sum of vertex degrees for a triangulation with $v$ vertices is $6v-12$, implying that $k$ can only take values from $\{2,3,4,5\}$, since for $k \ge 6$ the total sum of degrees would exceed this value.

The following statement was formulated without proof by Kempe in \cite{kem}: 
a spherical triangulation is $3$-vertex-colorable if and only if it is Eulerian, 
that is, all its vertex degrees are even. Note that the ``only if'' part is 
straightforward, so the core of the statement is that any Eulerian triangulation 
admits such a coloring. Moreover, Heawood claimed in \cite{hea}, also without proof, 
that the latter statement remains true if we weaken the assumption that $G$ is 
a triangulation to the condition that all face degrees of $G$ are divisible by $3$. 
According to a paper by Tsai and West \cite{tsai}, complete proofs of the above 
statements were achieved only in the twentieth century, apparently in the works 
by Golovina and Yaglom \cite{gol} and Steinberg \cite{stei}. Nowadays, many other 
proofs based on different ideas are known. We refer the reader to \cite{tsai} 
for a discussion of some of them.

Another problem related to graphs of this kind was posed by Eberhard in \cite{eber}: 
given a $3$-regular polyhedron $G$ such that its face degrees are all divisible 
by $3$, can it have an odd number of faces? Again, this simple-looking question 
was answered in the negative only in the twentieth century by Motzkin \cite{mot} 
and, via a different approach, by Kotzig \cite{kot}. The following more general 
dual result was proved by Gr\"unbaum \cite{gru} and, using another method, by Fisk \cite{fisk2}: 
let $G$ be a spherical triangulation such that the degrees of all its vertices 
are divisible by a number $k \in \{3,4,5\}$. Then the number of vertices $v$ 
of $G$ satisfies the following congruences:
\begin{align*}
\text{if } k = 3, & \quad \text{then }v \equiv 0 \pmod{2}, \\
\text{if } k = 4, & \quad \text{then }v \equiv 2 \pmod{4}, \\
\text{if } k = 5, & \quad \text{then } v\equiv 2 \pmod{10}.
\end{align*}

Our first result extends and refines the aforementioned findings to the following class of graphs. 
We say that a plane graph is \textit{uniform of type $(k,l)$} if all vertex degrees of $G$ are divisible by $k$ and all face degrees of $G$ are divisible by $l$. 
The Euler formula easily implies that, up to interchanging $k$ and $l$, the only possible pairs $(k,l)$ are $(2,r)$ with $r \geq 2$, $(3,3)$, $(3,4)$, and $(3,5)$. 
For a uniform graph of type $(k,l)$, we let $n(k,l)$ denote the number of edges of the unique Platonic graph that is uniform of type $(k,l)$. Furthermore, we denote by $v(k,l)$ and $f(k,l)$ the number of vertices and faces of this Platonic graph, respectively.  

In the above notation, we prove the following statement.

\begin{theorem} \label{t1}  
Let $G$ be a uniform graph of type $(k,l)$ with $n$ edges.
Then the following statements hold:
\begin{enumerate}[label={\rm \arabic*)}]
    \item The number 
    \[
    s=\frac{n}{n(k,l)}
    \]
    is an integer.   
    \item The graph $G$ is $v(k,l)$-vertex-colorable and $f(k,l)$-face-colorable. 
\end{enumerate}
Furthermore, there exist colorings as above such that the sum of the degrees of all vertices of any single color is equal to $ks$, and the sum of the degrees of all faces of any single color is equal to $ls$.
\end{theorem}

The first statement of the theorem implies the above congruences. Indeed, the number of vertices $v$ in a triangulation is related to the number of edges $n$ by the formula $v = 2 + n/3$. Thus, for instance, for $k=5$, recalling that the number of edges of the icosahedron is $30$, Theorem~\ref{t1} implies that $v = 2 + 10i$ for some $i \in \mathbb{N}$. However, the assumptions of Theorem~\ref{t1} are more general and also cover the case of type $(k,l) = (2,r)$ with $r \ge 2$. In particular, in the last case it implies that the number $n$ is divisible by $r$, while the handshaking lemma for faces implies only that $2n$ is divisible by $r$.

The second statement of the theorem provides a stronger version of the result concerning the vertex $3$-coloring of regular graphs of type $(2,3)$, ensuring not only the existence of a coloring, but also that its color distribution is uniform.

As noted above, our approach involves Belyi functions, and the complex-analytic counterpart of Theorem \ref{t1} is that a Belyi function $\beta$ of a uniform graph of type $(k,l)$ factors through the Belyi function $\beta_{k,l}$ of the unique Platonic graph of type $(k,l)$, that is, $\beta = \beta_{k,l} \circ \psi$ for some rational function $\psi$. Once this factorization is established, Theorem \ref{t1} readily follows from the basic algebraic and geometric properties of rational functions.

The next class of graphs considered in this paper consists of graphs whose vertex degrees are divisible by $k$ and whose face degrees are divisible by $l$, with at most two exceptions. These exceptions may occur among vertices or faces: there may be two exceptional faces, two exceptional vertices, or one exceptional vertex and one exceptional face. We shall refer to such graphs as $1$-uniform or $2$-uniform graphs of type $(k,l)$, depending on the total number of exceptions. We note that the parameters $k$ and $l$ for such graphs satisfy the same restrictions as for ordinary uniform graphs.

The first result concerning this class of graphs was obtained by Gr\"unbaum in 1967 \cite{gru}, who proved that if $G$ is a connected $3$-regular plane graph, and all but exactly $e \geq 1$ of its faces have degrees divisible by some $k \in \{2,3,4,5\}$, then $e$ cannot be equal to $1$. Moreover, if $e=2$, the two exceptional faces cannot share an edge. We note that Gr\"unbaum's result implies the non-existence of a fullerene  with 12 pentagons and only one hexagon (see also papers \cite{as}, \cite{buch}, \cite{deza15} employing alternative approaches to the proof of the latter fact). 

Essentially the same result as Gr\"unbaum's theorem, but in its dual form, was proved nearly fifty years later by Izmestiev \cite{izm}, who was apparently unaware of Gr\"unbaum's work. Namely, Izmestiev showed that if a triangulation of the sphere has exactly two vertices whose degrees are not divisible by some $k \in \{2,3,4,5\}$, then these vertices are not adjacent. He also established the non-existence of a triangulation in which all but exactly one vertex have degrees divisible by such $k$. The proofs in \cite{gru} and \cite{izm} are remarkably different: Gr\"unbaum's proof uses graph-theoretic constructions, whereas Izmestiev's relies on coloring theory and holonomy. In particular, Izmestiev employed a construction from Fisk's paper \cite{fisk2}. He also remarked that his result can be deduced from a lemma in two other papers by Fisk \cite{fisk1, fisk3}, combined with the four-color theorem.

All the above results are special cases of theorems in Malkevitch's book \cite{mal}, which is devoted to $1$- and $2$-uniform graphs of type $(k,l)$ with $k,l \in \{2,3,4,5\}$. Malkevitch showed that, with some exceptions, Gr\"unbaum's theorem remains true for such graphs, including the case where one allows one exceptional vertex and one exceptional face. The proof of these results is quite long and occupies the entirety of the book \cite{mal}. It relies on transformations that change the number of edges in certain faces while preserving their divisibility properties. 

Our second result recovers the theorems above and supplements them with divisibility restrictions on the possible number of edges of $2$-uniform graphs. It also includes the case $(k,l)=(2,r)$ with $r \geq 6$, which has seemingly not been considered previously.

\begin{theorem} \label{t2} 
The following statements hold:
\begin{enumerate}[label={\bf \Roman*.}]
\item There exists no $1$-uniform graph of any type.

\item If $G$ is a $2$-uniform graph of type $(k,l)$ of degree $n$ with two exceptional faces of degrees $u$ and $v$, then:
\begin{enumerate}[label={\rm \arabic*)}]
\item The equality $$\gcd(u,l) = \gcd(v,l)$$ holds.   
\item If $t$ is the number defined by the previous equality, then the number 
$$s = \frac{n}{n(k,l)}\cdot \frac{l}{t}$$  
is an integer.   
\item The exceptional faces are not adjacent unless $k=2$. 
\end{enumerate}

\item If $G$ is a $2$-uniform graph of type $(k,l)$ with $n$ edges having one exceptional vertex of degree $u$ and one exceptional face of degree $v$, then: 
\begin{enumerate}[label={\rm \arabic*)}]
\item The equality $k=l=3$ holds. 
\item The number $n$ is even.
\item The exceptional vertex is not incident to the exceptional face.
\end{enumerate}
\end{enumerate}
\end{theorem}

The proof of Theorem~\ref{t2} is also carried out using functions. However, it becomes more complicated, as the proof of restrictions on $k$ and $l$ involves using some non-trivial non-existence results for rational functions with certain ramification patterns. Another ingredient of the proof is the fact that the Belyi functions for the graphs under consideration factor through simpler Belyi functions, which allows us to obtain restrictions on the number $n$ and the relative positions of exceptional vertices and faces.

The last large family of graphs we consider in this paper consists of graphs where all vertex degrees equal a number $k$ and all face degrees equal a number $l$, with at most two exceptions that, as above, may occur among the vertices or among the faces. We shall call such graphs $1$-Platonic or $2$-Platonic of type $(k,l)$.

In \cite{kei1}, Keith, Froncek, and Kreher considered the question of whether a $1$-Platonic graph exists. Naturally, by Theorem~\ref{t1}, if $G$ is a $1$-Platonic graph of type $(k,l)$ with, say, one exceptional face, then its degree must be divisible by $l$, which however does not preclude the existence of such a graph. The authors noted that the statement regarding the non-existence of these graphs appears without proof in \cite{deza1} by Deza and Sikiri\'{c}, and with only a sketch of a proof in \cite{deza2} by Deza, Sikiri\'{c}, and Shtogrin. The authors of \cite{kei1} proved that a $1$-Platonic graph does not exist under the additional assumption that the graph is $2$-connected. In full generality, the non-existence of $1$-Platonic graphs was subsequently proved in \cite{fro2} by Froncek, Khorsandi, Musawi, and Qiu. The authors of \cite{kei1} also conjectured that if $G$ is a $2$-Platonic graph with two exceptional faces, then the degrees of these faces are equal. This conjecture was proved affirmatively in \cite{jen}.

The following result sums up the above findings and completes them for the case where one exceptional face and one exceptional vertex exist.

\begin{theorem} \label{t3} 
The following statements hold:
\begin{enumerate}[label={\rm \arabic*)}]
\item There exists no $1$-Platonic graph of any type. 
\item If $G$ is a $2$-Platonic graph of type $(k,l)$ with two exceptional faces of degrees $u$ and $v$, then $u = v$.
\item If $G$ is a $2$-Platonic graph of type $(k,l)$ with one exceptional vertex of degree $u$ and one exceptional face of degree $v$, then either $k = l = 3$ and $u = v$,  or $G$ is a uniform graph of type $(k,l)$ and $u/k = v/l$. 
\end{enumerate} 
\end{theorem}

The proof of Theorem~\ref{t3} is obtained by modifying the arguments from the proofs of Theorems~\ref{t1} and~\ref{t2}. In fact, our proof yields much more than merely the statement of the theorem, providing a simple constructive description of Belyi functions for $2$-Platonic graphs. This complements related results from papers \cite{deza2, fro3, jen}, where the classification of $2$-Platonic graphs was studied in the case of two exceptional faces.

This paper is organized as follows. In Section 2, we review the correspondence between plane graphs and Belyi functions, which links graph existence questions to the Hurwitz existence problem. In Section 3, within the framework of this problem, we reprove several results from \cite{pak} concerning the realizability of ramification patterns of rational functions. Specifically, Theorems~\ref{h1}--\ref{h6}, taken together, are equivalent to Theorem~1 and Theorem~2 from \cite{pak}. Nevertheless, the formulations and proofs presented in this paper are substantially less technical than those in \cite{pak}. Namely, rather than using tools and results related to orbifolds and fiber products, we rely on standard results from complex analysis concerning analytic continuation and monodromy.

In Section 4, we apply the results of Sections 2 and 3 to uniform graphs and prove Theorem~\ref{t1}. In Section 5, we prove the non-existence of $1$-uniform and $1$-Platonic graphs. In Section 6, we study $2$-uniform and $2$-Platonic graphs and, in particular, prove Theorems~\ref{t2} and~\ref{t3}. In Section 7, we solve the realization problem for {\it bicolored} $2$-Platonic graphs, thereby settling the classical Hurwitz existence problem for the corresponding class of Belyi functions. Note that Theorem~\ref{62} and Theorem~\ref{72} in Sections 6 and 7 are also special cases of the findings in \cite{pak} (namely, Theorem~3), but their formulations and proofs are given in more constructive terms, illustrating the applications of the general result of \cite{pak}.

Finally, in Section 8, we discuss several problems that may serve as natural next steps for further research.

\section{Correspondence between graphs and Belyi functions}
The theory of ``dessins d'enfants'', originating from the work of Alexander Gro\-the\-ndieck \cite{gro}, connects the topology of surfaces, algebraic geometry, and Galois theory. A \textit{dessin d'enfant} is defined as a 2-cell embedding of a bicolored graph into a compact oriented surface. Every such object uniquely determines, up to isomorphism, a \textit{Belyi pair} consisting of a Riemann surface $X$ representing an algebraic curve defined over $\overline{\mathbb{Q}}$, together with a meromorphic function $\beta: X \to \mathbb{C}\mathbb{P}^1$ ramified only over $\{0, 1, \infty\}$. This correspondence enables the absolute Galois group $\text{Gal}(\overline{\mathbb{Q}}/\mathbb{Q})$ to act naturally on dessins, which was one of the primary motivations for studying this theory. 

In this paper, we require only the topological aspect of the dessins d'enfants  theory, and only in a very special case concerning the correspondence between plane graphs and Belyi functions. For the reader's convenience, in this section, after providing the necessary definitions, we recall how this correspondence is constructed, as it is essential for the entirety of this work. We also provide a sketch of the proof of this correspondence, referring the reader for more details to any  introduction to the theory of dessins d'enfants~--- for example, the monograph~\cite{lz}. Note that any embedding of a connected graph into the sphere is well known to be a $2$-cell embedding, and thus falls under the definition of a dessin d'enfant given above. Thus, in the spherical case, dessins d'enfants are simply connected plane graphs, and we will use the latter term.

The correspondence between connected plane graphs and Belyi functions has two versions: one for bicolored graphs and another for ordinary graphs. We start with a description of the first version, where the primary topological objects are connected bicolored graphs embedded in the sphere (where as above, multiple edges and loops are allowed). We consider these objects up to an orientation-preserving isotopy equivalence that preserves the vertex coloring. 
The corresponding analytical objects are Belyi functions, i.e., rational functions whose critical values are contained in the set $\{0, 1, \infty\}$. Two such functions, $\beta$ and $\tilde{\beta}$, are considered equivalent if there exists a M\"{o}bius transformation $\mu$ such that $\beta = \tilde{\beta} \circ \mu$.

From a Belyi function $\beta$ of degree $n$, one can construct a connected bicolored graph $G_\beta$ with $n$ edges as follows. Consider the segment $[0, 1]$ as a bicolored graph with a black vertex at $0$ and a white vertex at $1$. Since the open interval $(0, 1)$ contains no critical values of $\beta$, the preimage $\beta^{-1}([0,1])$ consists of $n$ arcs that can intersect only at the preimages of $0$ and $1$. Thus, by treating the preimages of $0$ and $1$ as vertices and the connected components of $\beta^{-1}((0,1))$ as edges, we obtain a graph $G_\beta$ with $n$ edges embedded in the sphere (see Fig.~\ref{f1}, which represents a graph $G_{\beta}$ for some Belyi function $\beta$ of degree 9).  
Note that the degree of any vertex $x$ of $G_\beta$ coincides with the multiplicity $k$ of the corresponding point under $\beta$, since $\beta$ behaves like a power map $z \mapsto z^k$ near $x$.  

\begin{figure}[h]
    \centering
    \includegraphics[width=0.8\textwidth]{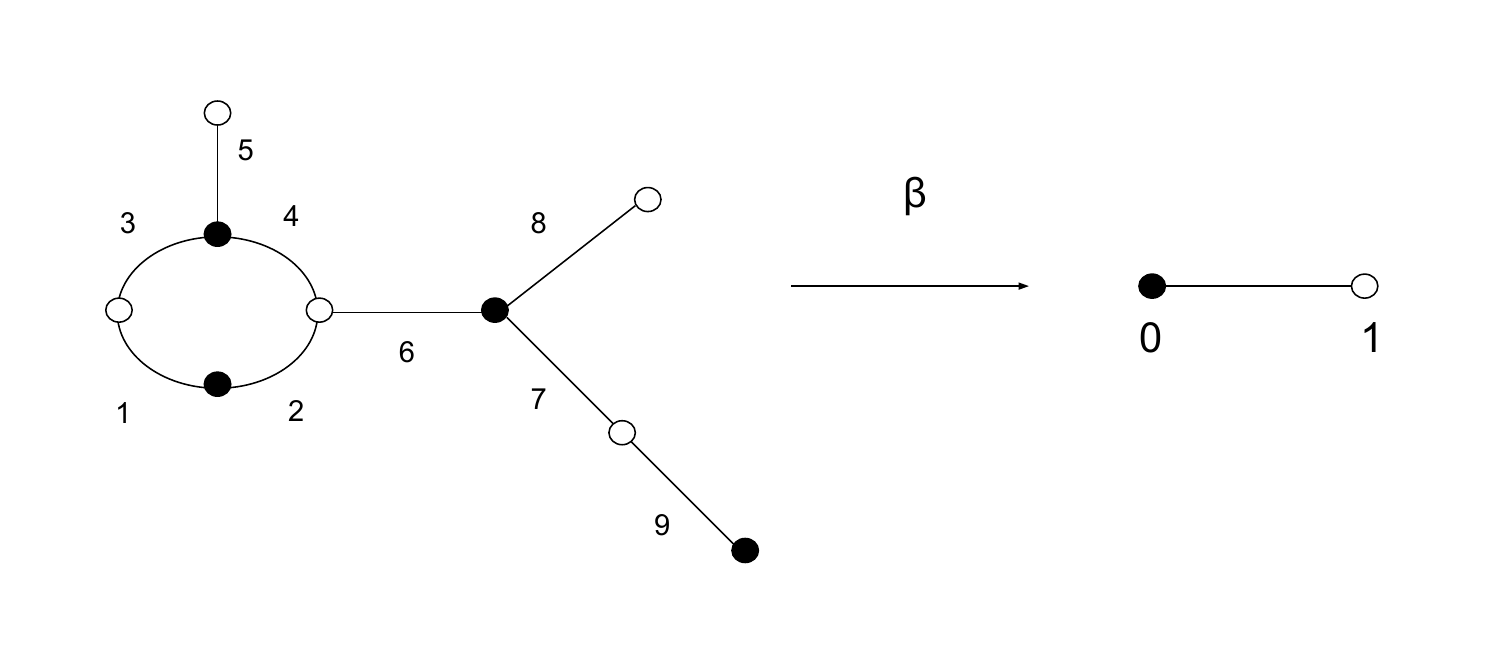}
    \caption{}
    \label{f1}
\end{figure}

It is easy to see that $G_\beta$ naturally inherits a bicolored structure by coloring all preimages of $0$ black and all preimages of $1$ white. Furthermore, since the complement $D = \mathbb{CP}^1 \setminus \big([0,1] \cup \{\infty\}\big)$ is topologically a punctured disk, where the boundary of the disk wraps twice around the segment, and the restriction 
\[
\beta: \beta^{-1}(D) \longrightarrow D
\]
is a covering map, the faces of $G_\beta$ are in bijective correspondence with the poles of $\beta$, and the multiplicity of each pole is equal to half the degree of the corresponding face. Finally,  it follows from the Riemann--Hurwitz formula that the number of vertices of $G_\beta$ plus the number of faces equals the number of its edges plus $2$. Thus, the graph $G_{\beta}$ is connected.

\bt \l{b1} 
The map $\beta\rightarrow G_{\beta}$ descends to a bijective correspondence between the equivalence classes of Belyi functions of degree $n$ and the equivalence classes of connected bicolored plane graphs with $n$ edges.
\et
\noindent{\it Sketch of the proof.} 
It is clear that equivalent Belyi functions lead to equivalent graphs. Thus, we must essentially show that any connected bicolored plane graph can be realized by a suitable Belyi function. This is achieved via the {\it monodromy representation}, as explained below.

Let $\beta$ be a Belyi function of degree $n$. Fix a base point $z_0 \in (0,1)$ and denote by $S = \beta^{-1}(z_0)$ the set of its $n$ preimages under $\beta$. The fundamental group \linebreak $\pi_1(\mathbb{C}\mathbb{P}^1 \setminus \{0,1,\infty\}, z_0)$ is generated by three loops $\gamma_0$, $\gamma_1$, and $\gamma_{\infty}$ encircling $0$, $1$, and $\infty$, respectively. Lifting these loops yields a triple of permutations $\sigma_0$, $\sigma_1$, and $\sigma_{\infty}$ acting on the set $S$ and satisfying the relation
\be \l{sa}
\sigma_0 \sigma_1 \sigma_\infty = \mathrm{id}.
\ee 
Conversely, by the general theory of coverings and Riemann's existence theorem, any permutations $\sigma_0$, $\sigma_1$, and $\sigma_\infty$ acting transitively on an $n$-element set, satisfying~\eqref{sa}, and fulfilling the genus-zero condition
\[
c(\sigma_0) + c(\sigma_1) + c(\sigma_\infty) = n + 2,
\]
where $c(\sigma)$ denotes the number of disjoint cycles of $\sigma$, uniquely determine a rational function $\beta: \mathbb{C}\mathbb{P}^1 \to \mathbb{C}\mathbb{P}^1$ ramified only over $\{0,1,\infty\}$ whose monodromy permutations coincide with $\sigma_0$, $\sigma_1$, and $\sigma_\infty$.

Given a Belyi function $\beta$, the permutations $\sigma_0$, $\sigma_1$, and $\sigma_{\infty} = (\sigma_0\sigma_1)^{-1}$ can be read directly from the graph $G_{\beta}$ as permutations acting on the edges of $G_{\beta}$. Indeed, $\sigma_0$ and $\sigma_1$ correspond to the permutations obtained by tracking the edges of $G_\beta$ counterclockwise around the white and black vertices, respectively. 
For example, for the graph shown in Fig.~\ref{f1}, these permutations are 
\[
\sigma_0=(12)(345)(678)(9),\quad \sigma_1=(13)(264)(5)
(79)(8).
\]
Note that the cycles of $\sigma_0$, $\sigma_1$, and $\sigma_{\infty}$ correspond to the white vertices, black vertices, and faces of $G_{\beta}$, respectively.

Since the above permutations acting on the edges of $G$ are defined solely by the geometry of $G$, this allows one to define them starting from any connected bicolored plane graph $G$. These permutations, in turn, yield a Belyi function with the matching monodromy representation, defined up to equivalence by M\"{o}bius precomposition. Therefore, for any graph $G$ we can obtain a Belyi function $\beta$ such that $G_{\beta}$ is equivalent to $G$. \qed

It is worth noting that, for a given graph, the explicit computation of the corresponding Belyi function is a highly non-trivial task, often requiring sophisticated theoretical and numerical methods. In the present paper, however, we do not need to construct these functions explicitly, as our arguments essentially rely merely on their existence.

Although not every graph is bipartite, this poses no obstacle to applying the theory of dessins d'enfants to arbitrary graphs. Indeed, with any ordinary graph with $n$ edges we can associate a bicolored graph with $2n$ edges by coloring all original vertices black and placing a white vertex at the midpoint of each edge (see Fig. \ref{f2}). The analytical object corresponding to this new graph is called a \textit{clean} Belyi function, and is defined as a Belyi function $\beta$ such that the multiplicity of any point in the preimage $\beta^{-1}(1)$ is exactly $2$. As above, we consider graphs up to isotopy equivalence and clean Belyi functions up to equivalence by M\"{o}bius precomposition.

\begin{figure}[h]
    \centering
    \includegraphics[width=0.8\textwidth]{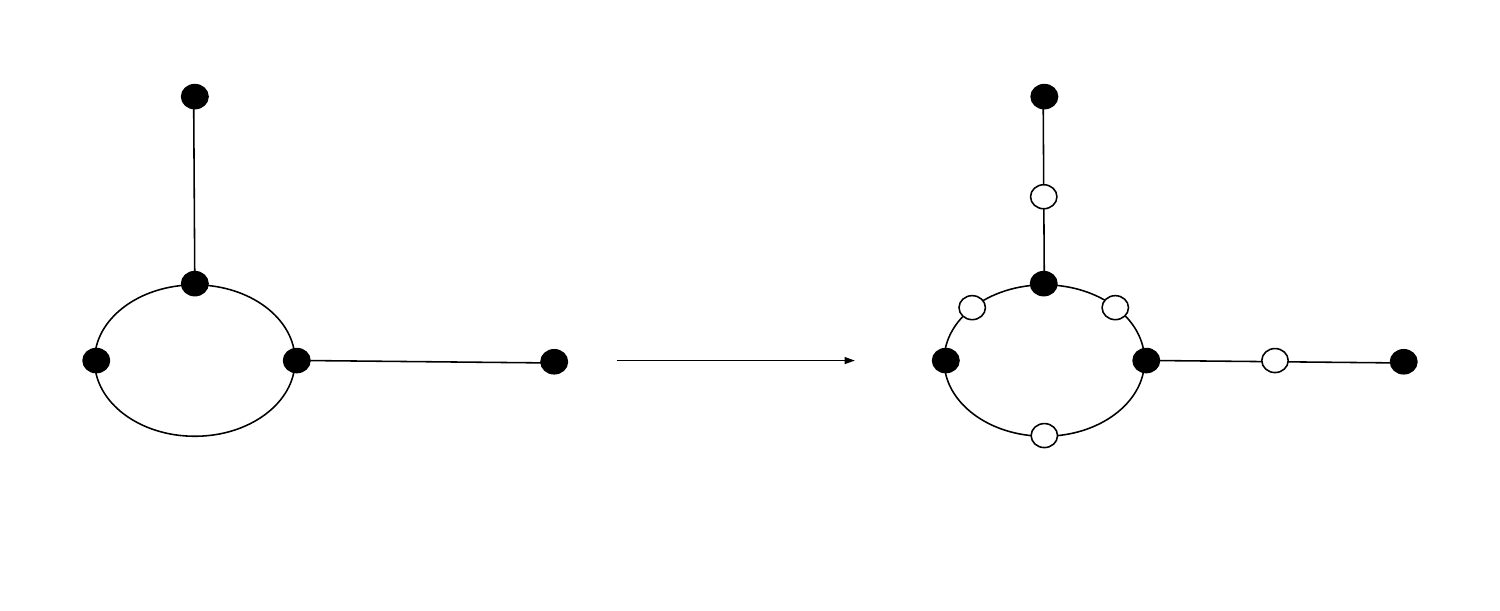}
    \caption{}
    \label{f2}
\end{figure}

Any clean function $\beta$ defines an \textit{ordinary} graph $G_{\beta}$ as follows. Considering the preimage of the segment $[0, 1]$ under a clean Belyi function $\beta$, we obtain a bicolored graph where every white vertex has degree $2$ and lies at the midpoint of an arc connecting two black vertices. By ``forgetting'' these white vertices, we obtain an ordinary graph $G_{\beta}$. Note that in distinction with biclored graphs the degree  a clean Belyi function $\beta$ corresponding to an ordinary graph with $n$ edges is $2n$, while the degree of each face of the graph  is equal to the multiplicity of the corresponding pole of $\beta$.

Theorem \ref{b1} implies the following statement.  

\bt \l{b2} 
The map $\beta \to G_{\beta}$ descends to a bijective correspondence between the equivalence classes of clean Belyi functions of degree $2n$ and the equivalence classes of connected plane graphs with $n$ edges. \qed
\et

Note that if $\beta$ is a clean Belyi function for a connected plane graph $G$, then the clean Belyi function for the dual graph is simply $1/\beta$. Indeed, clearly 
\[
(1/\beta)^{-1}([0, 1]) = \beta^{-1}([1, \infty]).
\]
On the other hand, as is easy to see, the preimage $\beta^{-1}([1, \infty])$ precisely represents the dual graph (see Fig. \ref{f3}).

\begin{figure}[h]
    \centering
    \includegraphics[width=0.9\textwidth]{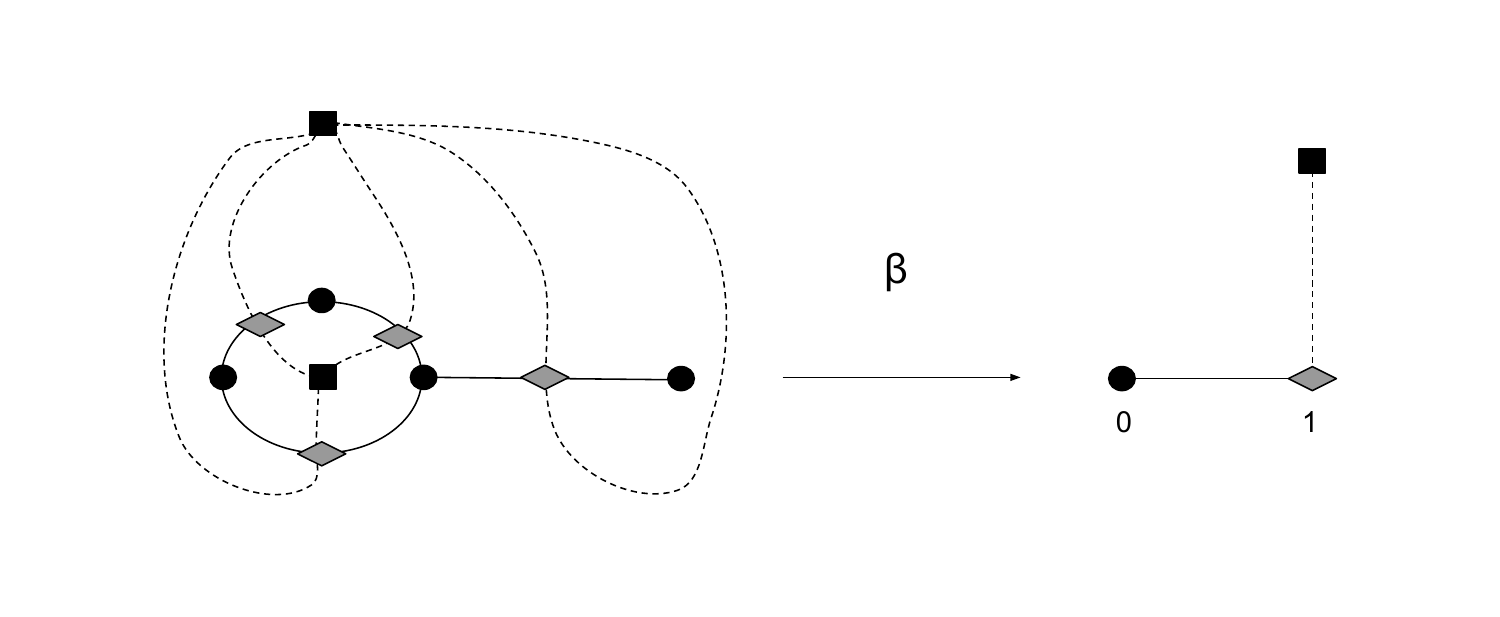}
    \caption{}
    \label{f3}
\end{figure}

Since we use both versions of the correspondence between graphs and Belyi functions in this paper, to avoid any confusion, we will always explicitly state whether the Belyi function under consideration is clean or the graph is bicolored. On the other hand, since all considered graphs are connected and plane, we will usually omit the corresponding adjectives.

With each bicolored plane graph $G$ with $n$ edges, one can associate three partitions of the integer $n$,
\begin{equation} \label{e}
\Pi_1 = (\pi_{1,1}, \dots, \pi_{1,p}), \quad \Pi_2 = (\pi_{2,1}, \dots, \pi_{2,q}), \quad \Pi_3 = (\pi_{3,1}, \dots, \pi_{3,r}),
\end{equation}
as follows: the first partition is the list of the black  vertex degrees, the second partition is the list of the white vertex degrees, and the third partition is the list of the face degrees divided by two. By Euler's formula, the parameters $p$, $q$, and $r$ are related by the equality 
\begin{equation} \label{eu}
p + q + r = n + 2.
\end{equation}

In a similar manner, one can associate three partitions $\Pi_1$, $\Pi_2$, and $\Pi_3$ with any Belyi function $\beta$, where the first, second, and third partitions are the lists of the multiplicities of $\beta$ at the points of $\beta^{-1}(0)$, $\beta^{-1}(1)$, and $\beta^{-1}(\infty)$, respectively. By the construction of the bicolored graph $G_\beta$ from the Belyi function $\beta$, these sets of partitions clearly coincide for $\beta$ and $G_{\beta}$. Thus, the existence or non-existence of a graph with a given triple of partitions is equivalent to the existence or non-existence of the corresponding Belyi function.

We will call any collection $\Pi=(\Pi_1,\Pi_2,\Pi_3,n)$, where $\Pi_1,\Pi_2,\Pi_3$ are partitions of $n$ as in~\eqref{e} satisfying~\eqref{eu}, a \textit{passport}. 
Given a bicolored graph or a Belyi function, the associated set of partitions will be called the \textit{passport of the graph} or of the Belyi function.
We will say that a passport $\Pi$ is \textit{realizable} or \textit{unrealizable} depending on whether there exists a bicolored graph $G$ or, equivalently, a Belyi function $\beta$ such that $\Pi$ is the passport of $G$ or $\beta$.

For example, we say that the passport of the bicolored graph in Fig.~\ref{f1} is 
$$((3,3,2,1), (3,2,2,1,1), (7,2), 9),$$ 
and that the passport 
$$((2,2), (2,2), (3,1), 4)$$ 
is not realizable (we leave the verification of this to the reader).

%Наконец для графа (не двукрашенного) удобно определить паспорт как псапорт соответсвенно двуукрашенно с удаленным разбиением \Pi_2 состоящим из двоек. Гапрмер для графа на левой части фиг \ref{f2} its passport is
%$$(3,3,2,1,1), (7,3), 10).$$

The problem of describing the realizable passports of Belyi functions is a special case of the more general Hurwitz existence problem. In its general setting, one considers a larger number of partitions that potentially correspond to the ramification patterns of a rational function (or, more generally, a covering map between surfaces), with parameters governed by the Riemann--Hurwitz formula. 
We note that with a function having more than three critical values, one can 
still associate a combinatorial object generalizing a bicolored graph, known 
as a \textit{constellation} (see, e.g.,~\cite{lz}). However, in this paper, 
we do not consider these objects.

In this work, we mainly consider ordinary graphs. As noted above, the set of such 
graphs with $n$ edges naturally embeds into the set of bicolored graphs with 
$2n$ edges. We will always think of our graphs as the corresponding bicolored ones, 
even though we will use the term vertex only for black vertices and 
draw 
figures without white midpoint vertices. In particular, we will associate with 
any ordinary graph the passport of the corresponding bicolored graph (whose second 
partition always consists of twos) and the clean Belyi function of the latter.

\section{Hurwitz existence problem and monodromy} 
Let $f$ be a rational function such that there exist points $z_1, z_2, z_3 \in \mathbb{C}\mathbb{P}^1$ and integers $k, m, l \geq 2$ satisfying the following conditions: all multiplicities of $f$ at the points of $f^{-1}(z_1)$ are divisible by $k$, all multiplicities of $f$ at the points of $f^{-1}(z_2)$ are divisible by $m$, and all multiplicities of $f$ at the points of $f^{-1}(z_3)$ are divisible by $l$. Then it follows easily from the Riemann--Hurwitz formula that 
\begin{equation} \label{ine} 
\frac{1}{k}+\frac{1}{m}+\frac{1}{l}>1.
\end{equation} 
Moreover, this inequality remains valid even if the conditions above hold for all points in $f^{-1}(\{z_1,z_2,z_3\})$ except for at most two points. 

The solutions to \eqref{ine} are well known: at least one of the integers must be $2$, and if, say, $m=2$, then, up to interchanging $k$ and $l$, the only possible pairs $(k,l)$ are $(2,r)$ with $r \geq 2$, $(3,3)$, $(3,4)$, and $(3,5)$. For brevity, we will say that a triple of integers $k, m, l \geq 2$ is \textit{Platonic} if it satisfies \eqref{ine}. We say that a pair of integers $k, l \geq 2$ is \textit{Platonic} if the triple $(k, 2, l)$ is Platonic.

For every solution to \eqref{ine} and any $z_1, z_2, z_3 \in \mathbb{C}\mathbb{P}^1$, there exists a rational function $f$ such that all multiplicities of $f$ at the points of $f^{-1}(z_1)$ are equal to $k$, all multiplicities of $f$ at the points of $f^{-1}(z_2)$ are equal to $m$, and all multiplicities of $f$ at the points of $f^{-1}(z_3)$ are equal to $l$. We will denote such a function by $\beta_{k,m,l}^{z_1,z_2,z_3}$.

One of the quickest ways to see that such a function exists is via the correspondence between graphs and clean Belyi functions.
Indeed, if $z_1=0$, $z_2=1$, $z_3=\infty$, and $m=2$, this correspondence implies that as $\beta_{k,m,l}^{z_1,z_2,z_3}$ one can take the clean Belyi function corresponding to the unique Platonic graph of type $(k,l)$. Moreover, $\beta_{k,m,l}^{z_1,z_2,z_3}$ is defined in a unique way up to post-composition with a M\"{o}bius transformation, and for different choices of $z_1, z_2, z_3$, the corresponding functions differ by pre-composition with a M\"{o}bius transformation. In particular, we have:
\be \l{deg} \deg \beta_{k,m,l}^{z_1,z_2,z_3}=2n(k,l).\ee

A more classical description of the functions $\beta_{k,m,l}^{z_1,z_2,z_3}$ is as invariant rational functions for the finite automorphism groups $D_{2r}$, $A_4$, $S_4$, and $A_5$ of $\mathbb{C}\mathbb{P}^1$, which serve as the symmetry groups of the Platonic solids of types $(2,r)$, $(3,3)$, $(3,4)$, and $(3,5)$, respectively. As such, they were already computed by Klein in his famous monograph~\cite{klein}.

Since the function $\beta_{k,m,l}^{z_1,z_2,z_3}$ is defined up to postcomposition with a Möbius transformation,  by writing $\beta_{k,m,l}^{z_1,z_2,z_3}$ we usually mean \textit{some} representative of the corresponding equivalence class. However, sometimes we will use this notation for a representative possessing desired properties.

The first result of this section is the following statement.

\bt \l{h1} 
Let $f$ be a rational function. Assume that there exist points $z_1, z_2, z_3$ on $\mathbb{C}\mathbb{P}^1$ and integers $k, m, l \geq 2$ such that: 
\begin{enumerate}[label={\rm \arabic*)}]
  \item All multiplicities of $f$ at the points of $f^{-1}(z_1)$ are divisible by $k$.
  \item All multiplicities of $f$ at the points of $f^{-1}(z_2)$ are divisible by $m$.
  \item All multiplicities of $f$ at the points of $f^{-1}(z_3)$ are divisible by $l$. 
\end{enumerate}
Then there exists a rational function $\psi$ such that $f = \beta_{k,m,l}^{z_1, z_2, z_3} \circ \psi$.   
In particular, $\deg f$ is divisible by $2n(k,l)$.
\et
\begin{proof} Choose a point $z_0$ such that $f(z_0) \notin \{z_1, z_2, z_3\}$. Then, by the Implicit Function Theorem, any branch of the algebraic function \be \l{mul} (\beta_{k,m,l}^{z_1,z_2,z_3})^{ -1} \circ f\ee defines a meromorphic germ at $z_0$, which can be analytically continued along any path avoiding the points in $f^{-1}(\{z_1, z_2, z_3\})$. Fixing one such germ $\psi_{z_0}$, we shall show that it extends to a global meromorphic function $\psi$ on $\mathbb{C}\mathbb{P}^1$, which will serve as the desired rational function.

To analyze the monodromy of the continuation of $\psi_{z_0}$, fix $z \in f^{-1}(z_1)$ and consider a composite loop $\gamma_z = \sigma \cdot \gamma \cdot \sigma^{-1}$ starting at $z_0$, where $\gamma$ is a small loop around $z$ and $\sigma$ is a path connecting $z_0$ to $\gamma$ in $\mathbb{C}\mathbb{P}^1 \setminus f^{-1}(\{z_1, z_2, z_3\})$. As we continue $\psi_{z_0}$ along $\gamma_{z}$, its image $f(\gamma_{z})$ winds around $z_1$ a number of times equal to the multiplicity of $f$ at $z$. Since this multiplicity is divisible by $k$ and the multiplicity of $\beta_{k,m,l}^{z_1,z_2,z_3}$ at any point of $(\beta_{k,m,l}^{z_1,z_2,z_3})^{-1}(z_1)$ is exactly $k$, the lift of the path $\gamma_z$ via \eqref{mul} closes up. 
Thus, analytically continuing the germ $\psi_{z_0}$ along $\gamma_z$ returns it to itself.

Since the same property holds for loops around points of $f^{-1}(z_2)$ and $f^{-1}(z_3)$, we conclude that the monodromy of the analytic continuation of $\psi_{z_0}$ is trivial. Therefore, by the Monodromy Theorem, the germ $\psi_{z_0}$ extends to a single-valued meromorphic function on the punctured sphere $\mathbb{C}\mathbb{P}^1 \setminus f^{-1}(\{z_1, z_2, z_3\})$. Observe now that since $f$ and $\beta_{k,m,l}^{z_1,z_2,z_3}$ are rational, the omitted points cannot be essential singularities for $\psi$. Hence, by the classification of isolated singularities, $\psi_{z_0}$ extends to a global meromorphic function $\psi$ on the entire Riemann sphere $\mathbb{CP}^1$, which must be rational. Moreover, the local equality $$f=\beta_{k,m,l}^{z_1,z_2,z_3} \circ \psi_{z_0}$$ extends globally to the equality $$f = \beta_{k,m,l}^{z_1,z_2,z_3} \circ \psi$$ by the uniqueness of analytic continuation. Finally, it is clear that 
\[
\deg f = \deg \beta_{k,m,l}^{z_1,z_2,z_3} \cdot \deg \psi. \qedhere
\] 
\end{proof}

The following three statements are proved by a modification of the proof of Theorem~\ref{h1}.

\bt \l{h2} 
There exists no rational function $f$ with the property that, for some points $z_1, z_2, z_3 \in \mathbb{C}\mathbb{P}^1$ and integers $k, m, l \geq 2$, the following conditions hold:
\begin{enumerate}[label={\rm \arabic*)}]
  \item All multiplicities of $f$ at the points of $f^{-1}(z_1)$ are divisible by $k$.
  \item All multiplicities of $f$ at the points of $f^{-1}(z_2)$ are divisible by $m$.
  \item All multiplicities of $f$ at the points of $f^{-1}(z_3)$, except for one, are divisible by $l$.
\end{enumerate}
\et 
\pr 
Considering the analytic continuation of the germ $\psi_{z_0}$ defined in the proof of Theorem~\ref{h1} and using the same reasoning, we see that the monodromy of this continuation is trivial at all points of $\mathbb{C}\mathbb{P}^1$, except possibly at the point $z \in f^{-1}(z_3)$ where $\deg_z f = l'$ is not divisible by $l$. Hence, 
 by the Monodromy Theorem, it must be trivial at $z$ as well. 
 
 On the other hand, as we continue $\psi_{z_0}$ along $\gamma_z$ as above, its image $f(\gamma_z)$ winds around $z_3$ exactly $l'$ times. Therefore, the permutation of the branches of \eqref{mul} has the same cyclic structure as the $l'$-th power of the permutation $\sigma$ of the branches of $(\beta_{k,m,l}^{z_1,z_2,z_3})^{-1}$ corresponding to lifting a small loop around $z_3$ under $\beta_{k,m,l}^{z_1,z_2,z_3}$. However, since $\sigma$ consists of cycles of length $l$, its $l'$-th power has no fixed points because $l \nmid l'$. Thus, we arrive at a contradiction.\qed

\bt \l{h3} 
Let $f$ be a rational function. Assume that there exist points $z_1, z_2, z_3$ on $\mathbb{C}\mathbb{P}^1$ and integers $k, m, l \geq 2$ such that: 
\begin{enumerate}[label={\rm \arabic*)}]
  \item All multiplicities of $f$ at the points of $f^{-1}(z_1)$ are divisible by $k$.
  \item All multiplicities of $f$ at the points of $f^{-1}(z_2)$ are divisible by $m$.
  \item All multiplicities of $f$ at the points of $f^{-1}(z_3)$ are divisible by $l$, except for two points with multiplicities $u$ and $v$.
\end{enumerate}
Then $\gcd(u,l)=\gcd(v,l).$
\et
\pr Let $x$  and $y$ be the points in $f^{-1}(z_3)$ where $\deg_{x} f = u$ and  \linebreak $\deg_{y} f = v$. Considering again the analytic continuation of the germ $\psi_{z_0}$ and using the Monodromy Theorem, we conclude that the permutations of the branches of \eqref{mul} induced by the continuations along $\gamma_{x}$ and $\gamma_{y}$ are inverse to each other. This implies that the $u$-th power of the permutation $\sigma$ has the same cycle structure as the $v$-th power of $\sigma$, which yields $\gcd(u,l)=\gcd(v,l).$ \qed 

\bt \l{h4} 
Let $f$ be a rational function. Assume that there exist points $z_1, z_2, z_3$ on $\mathbb{C}\mathbb{P}^1$ and integers $k, m, l \geq 2$ such that: 
\begin{enumerate}[label={\rm \arabic*)}]
  \item All multiplicities of $f$ at the points of $f^{-1}(z_1)$ are divisible by $k$.
  \item All multiplicities of $f$ at the points of $f^{-1}(z_2)$ are divisible by $m$, except for one point with multiplicity $u$.
  \item All multiplicities of $f$ at the points of $f^{-1}(z_3)$ are divisible by $l$, except for one point with multiplicity $v$.
\end{enumerate}
Then $m/\gcd(u,m)=l/\gcd(v,l).$
\et
\pr The proof is similar to the proof of Theorem~\ref{h3}. Namely, let $x \in f^{-1}(z_2)$ and $y\in f^{-1}(z_3)$ be points where $\deg_{x} f=u$ and $\deg_{y} f=v$. Further, let $\sigma_1$ and $\sigma_2$ be the permutations of the branches of $(\beta_{k,m,l}^{z_1,z_2,z_3})^{-1}$ corresponding to lifting a small loop around $z_2$ and $z_3$, respectively, under $\beta_{k,m,l}^{z_1,z_2,z_3}$. Then the Monodromy Theorem implies that the $u$-th power of $\sigma_1$ has the same cycle structure as the $v$-th power of $\sigma_2$, which yields the desired equality. \qed

We complete this section with the following two results, which can be viewed as degenerate cases of the previous ones. 

\bt \l{h5} 
Let $f$ be a rational function. Assume that there exist points $z_1, z_2$ on $\mathbb{C}\mathbb{P}^1$ and an integer $k \geq 2$ such that all multiplicities of $f$ at the points of $f^{-1}(\{z_1,z_2\})$ are divisible by $k$. 
Then there exist a M\"obius transformation $\mu$ and a rational function $\psi$ such that $f = \mu \circ z^k \circ \psi$.
\et

\pr Assuming without loss of generality that $z_1 = 0$ and $z_2 = \infty$, the proof can be obtained in the spirit of the previous ones by considering the analytic continuation of a branch of the algebraic function $(z^k)^{-1} \circ f$. However, the use of analytic continuation is unnecessary in this case, as we can simply represent $f$ in terms of its zeros and poles. 
\qed

\bt \l{h6} 
Let $f$ be a rational function. Assume that there exist points $z_1, z_2$ on $\mathbb{C}\mathbb{P}^1$ and an integer $k \geq 2$ such that all multiplicities of $f$ at the points of $f^{-1}(\{z_1,z_2\})$ are divisible by $k$, except possibly for two points with multiplicities $u$ and $v$. 
Then $\gcd(u,k)=\gcd(v,k).$
\et

\pr As above, the statement is obtained using analytic continuation or elementary reasoning. In fact, assuming that $z_1 = 0$, $z_2 = \infty$, and that the points where the multiplicities of $f$ are exceptional are also $0$ and $\infty$, it is easy to see that a function $f$ satisfies the condition of the theorem if and only if it has the form 
$
f = z^a R(z^n),
$ where $n$ is not divisble by $ a$.
\qed

\section{Uniform graphs}
In this section, we apply the results of the previous two sections to uniform graphs 
and prove Theorem~\ref{t1}. 
We will denote by $\beta_{k,l}$ the clean Belyi function for the unique Platonic graph of type $(k,l)$. 

The following statement is the graph-theoretic counterpart of Theorem~\ref{h1}.

\begin{theorem} \label{41} 
Let $G$ be a uniform  graph of type $(k,l)$ with $n$ edges, and let $\beta$ be its clean Belyi function. Then there exists a rational function $\psi$ such that $\beta = \beta_{k,l} \circ \psi.$ 
In particular, $n$ is divisible by $n(k,l)$.
\end{theorem}

\begin{proof}
By the correspondence between graphs and Belyi functions, if $\beta$ is a clean Belyi function of $G$, then all multiplicities of $\beta$ at the points of $\beta^{-1}(0)$ are divisible by $k$, all multiplicities of $\beta$ at the points of $\beta^{-1}(1)$ are divisible by $2$, and all multiplicities of $\beta$ at the points of $\beta^{-1}(\infty)$ are divisible by $l$. Therefore, by Theorem~\ref{h1}, the equality 
\[
\beta = \beta_{k,2,l}^{0,1,\infty} \circ \psi = \beta_{k,l} \circ \psi
\]
holds. In particular, 
\[
 2n(k,l)=\deg \beta_{k,l}  \mid \deg \beta = 2n. \qedhere
\]
\end{proof}

Let $\hat{\beta}$ be a clean Belyi function and $\hat{G} = \hat{\beta}^{-1}([0,1])$ the corresponding graph. Assume that $\psi$ is a rational function such that the composition $\beta = \hat{\beta} \circ \psi$ is also a clean Belyi function, and let $G = \beta^{-1}([0,1])$ be the corresponding graph. Then the vertices of $G$ are the preimages of the vertices of $\hat{G}$ under $\psi$. Moreover, the set of poles of $\beta$, identified with the set of faces of $G$, is the preimage under $\psi$ of the set of poles of $\hat{\beta}$, identified with the set of faces of $\hat{G}$. 
Although completely describing the geometry of $G$ starting from $\hat{G}$ and $\psi$ is a non-trivial task, even in the case where $\psi$ itself is a Belyi function (see \cite{marr}), to derive Theorem~\ref{t1} from Theorem~\ref{41}, it is enough to prove the following simple statement.

\begin{lemma}\label{hp} 
Let $\beta$ and $\hat{\beta}$ be clean Belyi functions, and let $G = \beta^{-1}([0,1])$ and $\hat{G} = \hat{\beta}^{-1}([0,1])$ be the corresponding graphs. Assume that there exists a rational function $\psi$ such that $\beta = \hat{\beta} \circ \psi$. Then the following conditions hold:
\begin{enumerate}[label={\rm \arabic*)}]
  \item If $e$ is an edge connecting vertices $a$ and $b$ in $G$, then $\psi(e)$ is an edge connecting vertices $\psi(a)$ and $\psi(b)$ in $\hat{G}$.
  \item If $\gamma$ is a path connecting a vertex $a$ of $G$ to a pole $p$ of $\beta$ lying in a face of $G$ incident to $a$, then $\psi(\gamma)$ is a path connecting the vertex $\psi(a)$ of $\hat{G}$ to the pole $\psi(p)$ of $\hat{\beta}$ lying in a face of $\hat{G}$ incident to $\psi(a)$.
\end{enumerate}
\end{lemma}

\begin{proof}
Let us consider $e$ as a path running from $a$ to $b$ in $G$. Clearly, $\psi(e)$ is then a path running from $\psi(a)$ to $\psi(b)$ in $\hat{G}$. Moreover, this path contains no vertices of $\hat{G}$ distinct from $\psi(a)$ and $\psi(b)$, because otherwise the open arc between $a$ and $b$ would contain a point belonging to $\beta^{-1}(\{0\})$. Since any such point is a vertex of $G$, we obtain a contradiction with the assumption that $e$ is an edge. Thus, $\psi(a)$ and $\psi(b)$ are connected by an edge in $\hat{G}$. This proves the first statement.

To prove the second statement, we observe that since $\gamma$ connects $a$ to $p$, the path $\psi(\gamma)$ connects $\psi(a)$ to the pole $\psi(p)$ of $\hat{\beta}$. This path contains no points belonging to $\hat{\beta}^{-1}((0,1))$ because otherwise the original path $\gamma$ would contain points from $\beta^{-1}((0,1))$, which form the edges of $G$. Thus, $\psi(\gamma)$ does not cross the edges of $\hat{G}$, which means that $\psi(a)$ is incident to the face containing $\psi(p)$.
\end{proof}

\noindent{\it Proof of Theorem~\ref{t1}.} The first statement follows directly from Theorem~\ref{41}. 
To construct the vertex coloring for the second statement, we consider the Platonic graph of type $(k,l)$ as $\hat{G}=\beta_{k,l}^{-1}([0,1])$, take a proper $v(k,l)$-coloring of $\hat{G}$, and pull it back to $G=\beta^{-1}([0,1])$ via the function $\psi$ from the equality 
\begin{equation} \label{ho} 
\beta = \beta_{k,l} \circ \psi
\end{equation} 
provided by Theorem~\ref{41}.
Explicitly, each vertex $v$ of $G$ receives the color of its image $\psi(v)$.

If $a$ and $b$ are two adjacent vertices of $G$ and $e$ is the edge between them, then the assumption that $a$ and $b$ have the same color implies that $\psi(a) = \psi(b)$. By Lemma~\ref{hp}, this leads to the conclusion that $\psi(e)$ is a loop. Since $\hat{G}$ contains no loops, this proves that the vertex coloring is proper. The part about color distribution follows from equality~\eqref{ho}. 

The statement regarding the face coloring of $G$ is proved analogously by using duality. \qed

\section{Non-existence of 1-uniform and 1-Platonic graphs} 
In this short section, we prove the following result.

\begin{theorem}\label{51}
There exist neither $1$-uniform nor $1$-Platonic graphs of any type.
\end{theorem}

\begin{proof}
Assume that there exists a $1$-uniform graph $G$ of type $(k,l)$. By duality, it is enough to consider the case of a single exceptional face. Let $\beta$ be the clean Belyi function of $G$. Then all multiplicities of $\beta$ at the points of $\beta^{-1}(0)$ are divisible by $k$, all multiplicities of $\beta$ at the points of $\beta^{-1}(1)$ are divisible by $2$, and all multiplicities of $\beta$ at the points of $\beta^{-1}(\infty)$ except one are divisible by $l$. However, such a function cannot exist by Theorem~\ref{h2}. 

Assume now that $G$ is a $1$-Platonic graph of type $(k,l)$ with one exceptional face of degree $u$, and let $\beta$ be its Belyi function. Then, by the first part of the theorem, $u$ is divisible by $l$. Therefore, by Theorem \ref{41}, there exists a rational function $\psi$ such that equality \eqref{ho} holds. 
Since $G$ is $1$-Platonic by assumption, by comparing the multiplicities of the left- and right-hand sides of \eqref{ho} via the chain rule, we conclude that $\psi$ has a unique critical point, and hence a unique critical value. Since the Riemann–Hurwitz formula implies that a rational function with only one critical value cannot exist, we conclude that a $1$-Platonic graph also cannot exist. \qedhere
\end{proof}

\section{2-uniform graphs and 2-Platonic graphs} 
In this section, we study $2$-uniform and $2$-Platonic graphs. In particular, we prove Theorems~\ref{t2} and~\ref{t3}. We also prove an analogue of Theorem~\ref{41} for $2$-uniform graphs. By duality, it suffices to consider graphs with either two exceptional faces, or with one exceptional face and one exceptional vertex. 

We begin with the following result, which proves parts (1) and (2) of statements II and III in Theorem~\ref{t2}.

\begin{theorem} \label{61}
The following statements hold:
\begin{enumerate}[label={\bf \Roman*.}]
\item If $G$ is a $2$-uniform graph of type $(k,l)$ with $n$ edges having two exceptional faces of degrees $u$ and $v$, then:
\begin{enumerate}[label={\rm \arabic*)}]
\item The equality $\gcd(u,l) = \gcd(v,l)$ holds.
\item If $t$ is the integer defined by the previous equality, then the number 
\[
s = \frac{n}{n(k,l)} \cdot \frac{l}{t}
\]
is an integer.
\end{enumerate}

\item If $G$ is a $2$-uniform graph of type $(k,l) $ with $n$ edges having one exceptional vertex of degree $u$, and one exceptional face of degree $v$, then:
\begin{enumerate}[label={\rm \arabic*)}]
\item The equality $k=l=3$ holds. 
\item The number $n$ is even. 
\end{enumerate}
\end{enumerate}
\end{theorem} 

\begin{proof}
The first statement of Part I follows directly from Theorem~\ref{h3} applied to a clean Belyi function corresponding to $G$.     
To prove the second statement, we choose a clean Belyi function $\beta$ in the equivalence class corresponding to $G$ such that the poles of $\beta$ corresponding to the exceptional faces are located at $0$ and $\infty$. Then the chain rule implies that the composition $\beta \circ z^{l/t}$ is a clean Belyi function corresponding to a uniform graph of type $(k,l)$. Hence, by Theorem~\ref{41}, there exists a rational function $\phi$ such that the diagram
\begin{equation} \label{oh}
\begin{tikzcd}[row sep=2.7em, column sep=3.5em, baseline=(current bounding box.center)]
\mathbb{C}\mathbb{P}^1 \arrow[r,"\phi"] \arrow[d,"z^{l/t}"']    
    & \mathbb{C}\mathbb{P}^1 \arrow[d,"\beta_{k,l}"] \\
\mathbb{C}\mathbb{P}^1 \arrow[r,"\beta"']    
    & \mathbb{C}\mathbb{P}^1
\end{tikzcd}
\end{equation}
commutes. Since this implies the equality   
\begin{equation} \label{deg_eq}
n\cdot \frac{l}{t} = n(k,l)\cdot \deg \phi,
\end{equation}
the second statement of Part I follows.

To prove the first statement of Part II, we observe that if $G$ has one exceptional vertex of degree $u$ and one exceptional face of degree $v$, then by Theorem~\ref{h4} (with $k$ and $m$ interchanged), the equality 
\begin{equation}
\frac{k}{\gcd(u,k)} = \frac{l}{\gcd(v,l)}
\end{equation}
holds. However, it is easy to see that this equality cannot hold for $\{k,l\} = \{3,4\}$ or $\{3,5\}$. On the other hand, for $\{k,l\} = \{2,r\}$ with $r \geq 2$, it is possible only if both sides are equal to $2$. Thus, if, say, $k=2$, then $u$ must be odd, which implies that the sum of the degrees of all vertices is also odd---a contradiction. This shows that $(k,l) = (3,3)$.

Finally, the second statement of Part II can be obtained in the same way as in Part I by observing that if a clean Belyi function $\beta$ for $G$ is chosen such that the points of $\beta^{-1}(1)$ and $\beta^{-1}(\infty)$ corresponding to the exceptional vertex and face are located at $0$ and $\infty$, then the composition $\beta \circ z^{3}$ is a Belyi function corresponding to a uniform graph of type $(3,3).$ 
\end{proof}

For each Platonic pair $(k,l)$ with $k,l \in \{3,4,5\}$, let us introduce the graph $G_{k,l}$ as shown in Fig.~\ref{f4}, and denote by $\beta_{G_{k,l}}$ the corresponding clean Belyi function. Note that all these graphs are $2$-uniform, and their exceptional faces (or, in the case of $G_{3,3}$, a face and a vertex) have degree one.

\begin{figure}[h]
    \centering
    \includegraphics[width=0.75\textwidth]{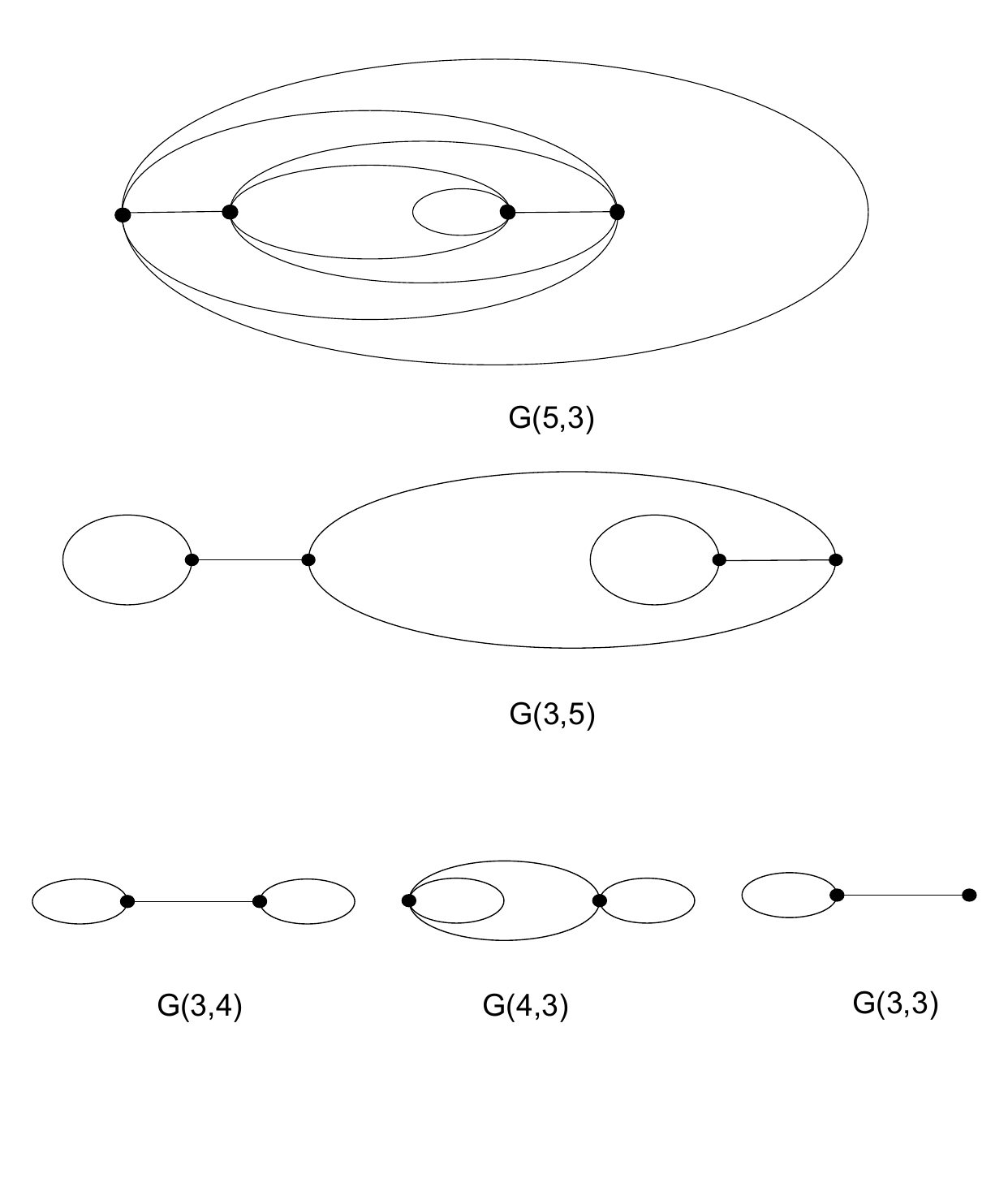}
   \caption{}
    \label{f4}
\end{figure}

The following statement is an analogue of Theorem~\ref{62} for $2$-uniform graphs of type $(k,l)$ with $k,l \in \{3,4,5\}$. Note that a similar analogue also exists for $2$-uniform graphs of other types. However, the most natural formulation of this generalization is for bicolored graphs (see Theorem~\ref{72} below). The reason for this is that a decomposition of a clean function $\beta$ into a composition $\beta = \hat{\beta} \circ \psi$ implies, in general, only that $\hat{\beta}$ is a Belyi function, which is not necessarily clean. This shows that the consideration of bicolored graphs may be necessary even when studying ordinary ones.

\begin{theorem} \label{62} 
Let $G$ be a $2$-uniform graph of type $(k,l)$ with $k,l \in \{3,4,5\}$, and let $\beta$ be its clean Belyi function. 
\begin{enumerate}[label={\rm \arabic*)}]
\item If $G$ has two exceptional faces, then there exists a rational function $\psi$ such that $\beta = \beta_{G_{k,l}} \circ \psi.$ 
\item If $G$ has one exceptional vertex and one exceptional face, then $k=l=3$ and there exists a rational function $\psi$ such that $\beta = \beta_{G_{3,3}} \circ \psi.$ 
\end{enumerate} 
\end{theorem}

\begin{proof}
The proof is obtained by a modification of the proof of Theorem \ref{t1}. Let us consider, for instance, the case $(k,l)=(5,3)$. The remaining cases are treated similarly.

Let $x$ and $y$ be the points in $\beta^{-1}(\infty)$ where $\deg_{x} \beta = u$ and $\deg_{y} \beta = v$. 
Choose a point $z_0$ such that $\beta(z_0) \notin \{0, 1, \infty\}$ and consider the algebraic function $(\beta_{G_{k,l}})^{-1} \circ \beta$. Arguing as in Section 3, we see that any branch of this function defines a meromorphic germ at $z_0$, which can be analytically continued along any path avoiding the points in $\beta^{-1}(\{0, 1, \infty\})$. Furthermore, since the permutation $\sigma$ of the branches of $(\beta_{G_{k,l}})^{-1}$ corresponding to lifting a small loop around $\infty$ under $\beta_{G_{k,l}}$ has the cyclic structure 
$$(3,3,3,3,3,3,1,1),$$ 
there is a germ $\psi_{z_0}$ that remains unchanged under analytical continuation along the path $\gamma_{x}$ defined as in Section 3. Since $\psi_{z_0}$ also remains unchanged under analytic continuation along the path $\gamma_z$ for any $z \in \beta^{-1}(\{0, 1, \infty\})$, except possibly $y$, for the same reasons as in the proof of Theorem~\ref{t1}, we conclude that the analytic continuation of $\psi_{z_0}$ has no monodromy at all. Hence, $\psi_{z_0}$ extends to a rational function $\psi$.
\end{proof}

Note that the chain rule implies that the function $\psi$, whose existence is guaranteed by Theorem~\ref{62}, maps the poles of $\beta$ corresponding to exceptional faces of $G$ to the poles of $\beta_{G_{k,l}}$ corresponding to exceptional faces of $G_{k,l}$. This can also be seen from the construction of $\psi$.

\vskip 0.2cm
\noindent{\it Proof of Theorem \ref{t2}.} By Theorems~\ref{51} and~\ref{61}, we only need to prove the statements regarding the relative position of exceptional vertices and faces. 
Let us consider a $2$-uniform graph $G = \beta^{-1}([0,1])$, for instance, of type $(5,3)$, and assume that its exceptional faces $x$ and $y$ are adjacent. By Theorem \ref{62}, there exists a rational function $\psi$ such that 
$$\beta = \beta_{G_{5,3}} \circ \psi.$$ 
Moreover, the images $\psi(x)$ and $\psi(y)$ are poles of $\beta_{G_{5,3}}$ lying in the exceptional faces of $G_{5,3}$, which we denote by $x'$ and $y'$.

Taking $e$ to be the edge in the dual graph $G^*$ connecting $x$ and $y$, we see by Lemma~\ref{hp} that the image $\psi(e)$ is an edge connecting $\psi(x)$ and $\psi(y)$ in the dual graph $G_{5,3}^*$. If the image of the set $\{x,y\}$ under $\psi$ consists of a single point $\psi(x)$ or $\psi(y)$, we obtain a contradiction because there is no loop in $G_{5,3}^*$ starting and ending at $x'$ or $y'$. On the other hand, if the images of the points $x$ and $y$ under $\psi$ are distinct, we also obtain a contradiction since the exceptional faces $x'$ and $y'$ of $G_{5,3}$ are not connected by an edge in $G_{5,3}^*$.

The remaining cases with $k,l \in \{3,4,5\}$ are treated similarly. 
Let us now consider the case of type $(k,2)$, where $k > 2$. It is more convenient to consider the dual situation. So assume that there exists a graph $G$ such that all of its vertices, except for two adjacent ones, have even degrees, while the degrees of all of its faces are divisible by $l > 2$, and show that this leads to a contradiction. 

Let us remove the edge $e$ connecting the vertices of odd degree. Then the resulting graph $G'$ has only vertices of even degree. Furthermore, if $e$ is incident to two distinct faces of degrees $L_1$ and $L_2$ in $G$, that is, if $G'$ is connected, the degrees of all faces of $G'$ are divisible by $l$, except for one face of degree $L_1 + L_2 - 2$. This latter degree is obviously not divisible by $l$, since otherwise $l > 2$ would be a divisor of $2$. However, such a graph cannot exist by Theorem~\ref{51}. 

Finally, if $e$ is incident to a single face of degree $L$ in $G$, that is, if $G'$ consists of two connected components, then the outer faces of these components have degrees $L_1'$ and $L_2'$ such that 
\[
L_1' + L_2' = L - 2.
\]
Hence, the degree of at least one of these faces is not divisible by $l$, which again contradicts Theorem~\ref{h2}. \qed

\begin{proof}[Proof of Theorem~\ref{t3}]
The first part of the theorem follows from Theorem~\ref{51}. To prove the remaining statements, we consider two cases: when the graph under consideration is not uniform, and when it is.

In the first case, arguing as in the proof of Theorem~\ref{61}, we see that there exists a rational function $\psi$ such that the equality 
\begin{equation} \label{ooh}
\beta \circ  z^{l/t} = \beta_{k,l} \circ  \psi
\end{equation}
holds. Moreover, since $G$ is $2$-Platonic, the chain rule applied to both sides of this equality implies that $\psi$ has exactly two critical points, which are located at zero and infinity.

It is easy to see that any function $\psi$ with this property has the form $\mu \circ z^s$ for some integer $s$ and a Möbius transformation $\mu$, and  therefore \eqref{ooh} takes the form 
\begin{equation} \label{oh-}
\beta \circ   z^{l/t}= \beta_{k,l} \circ z^{s}
\end{equation}
for some representative $\beta_{k,l}$ of the corresponding equivalence class of clean Belyi functions.

Calculating now the multiplicities of the left and right sides of \eqref{oh-} at zero and infinity, we obtain the equalities 
\[
u \cdot  \frac{l}{t} =s \cdot l, \qquad v \cdot \frac{l}{t}  =  s \cdot l
\]
implying that 
\begin{equation} \label{eqli}
u = v = s\cdot t=\deg \psi\cdot t=\frac{n \cdot l}{n(k,l)}.
\end{equation}

In the case where $G$ is uniform, we apply similar reasoning to the equality \eqref{ho} provided by Theorem~\ref{41}. Again, we conclude that without loss of generality we may assume that $\psi = z^s$. If $0$ and $\infty$ both belong to $\beta_{k,l}^{-1}(\infty)$, then the chain rule implies that $u = ls$ and $v = ls$. On the other hand, if one of these points belongs to $\beta_{k,l}^{-1}(\infty)$ and the other belongs to $\beta_{k,l}^{-1}(1)$, it implies that $u = ks$ and $v = ls$.
\end{proof}

\section{bicolored 2-uniform graphs and 2-Platonic graphs} 

Up to this point, we have considered ordinary graphs. However, the methods of \cite{pak}, refined in the present work, also apply to bicolored graphs, and in particular to the existence problem for such graphs. In turn, this last problem directly corresponds to the Hurwitz problem for Belyi functions. In this section, we demonstrate this by obtaining analogues of the results of the previous section for bicolored graphs. In particular, we completely describe the realizable passports of bicolored $2$-Platonic graphs.

For brevity, we will use the following conventions: the symbol $k^s$ within a passport will denote a string consisting of $s$ entries of $k$, while the symbol $k^*$ will denote a string consisting of an unspecified number of entries of $k$. Similarly, $\overline{k}^*$ will denote a string consisting of $s$ numbers divisible by $k$, where $s$ is unspecified. However, it will always be assumed that these unspecified numbers take values such that the resulting collection forms a passport. For example, the notation 
$$((4^*,2^2),(2^*),(\overline{3}^*),n)$$ 
means that the first partition consists of an unspecified number of $4$s along with two $2$s, the second partition consists exclusively of $2$s, and the third partition consists only of integers divisible by $3$, where the undefined parameters are chosen such that the sum of elements in each of the three partitions is equal to $n$ and, in addition, the total number of elements in all three partitions is $n+2$.

For any Platonic triple $(k,m,l)$, we denote by $\beta_{k,m,l}$ the Belyi function of the unique bicolored graph with the passport $(k^*,m^*,l^*)$, and set 
\begin{equation}
n(k,m,l) = \deg \beta_{k,m,l}.
\end{equation}
Of course, one of the numbers $k,m,l$ is still equal to $2$, but we no longer assume that this number is $m$. Therefore, a symmetric definition is required.

We define a bicolored uniform graph of type $(k,m,l)$ as a graph with a passport of the form 
\begin{equation}
(\overline{k}^*,\overline{m}^*,\overline{l}^*,n).
\end{equation}
Bicolored $2$-uniform and $2$-Platonic graphs are defined analogously to ordinary $2$-uniform and $2$-Platonic graphs, with the distinction that exceptional vertices can occur simultaneously among both black and white vertices, whereas in an ordinary graph (viewed as a bicolored one) the white vertices always have degree two and cannot be exceptional. By convention, we will not consider the Platonic graph of type $(2,2,r)$, $r \ge 2$, with the passport $((2^r),(2^r),(r^2),2r)$ as a $2$-Platonic graph of type $(2,2,m)$,  where  $m$ is not divisible by $r$.

The following result follows from Theorem \ref{h1} in the same way as Theorem \ref{41} does.

\begin{theorem} \label{71} 
Let $G$ be a bicolored plane graph with $n$ edges that is uniform of type $(k,m,l)$, and let $\beta$ be its Belyi function. Then there exists a rational function $\psi$ such that 
$
\beta = \beta_{k,m,l} \circ \psi.
$
In particular, the number of edges of $G$ is divisible by $n(k,m,l)$. \qed 
\end{theorem}

If $G$ is a $2$-Platonic graph that is uniform, an argument similar to the one used in the proof of Theorem~\ref{t3} shows that under a convenient normalization of $\beta_{k,m,l}$, the rational function $\psi$ appearing in Theorem~\ref{71} is simply $z^s$ for some $s \ge 2$. On the other hand, if $\beta_{k,m,l}$ is normalized in such a way that $0$ and $\infty$ belong to $\beta_{k,m,l}^{-1}(\{0,1,\infty\})$, then the composition $\beta_{k,m,l} \circ z^s$ is obviously a Belyi function corresponding to a uniform $2$-Platonic graph. Thus, we obtain the following corollary.

\begin{corollary} \label{c71}
A passport of a bicolored $2$-Platonic graph that is also a uniform graph is realizable if and only if it is of the form 
\[
\left((k^*),(m^*),(l^*,(ls)^2),n(k,m,l)\cdot s\right) \quad \text{or} \quad \left((k^*),(m^*,ms),(l^*,ls),n(k,m,l)\cdot s\right),
\]
where $(k,m,l)$ is a Platonic triple and $s\geq 2$ is an integer. \qed 
\end{corollary}

The following result is an analogue of Theorem \ref{62}.

\begin{theorem} \label{72}
Let $G$ be a bicolored $2$-uniform graph and $\beta$ be its Belyi function. Then there exists a rational function $\psi$ such that $\beta = \hat{\beta} \circ \psi$, where $\hat{\beta}$ is the Belyi function of a $2$-uniform graph with a passport from the following list:
\[
\begin{array}{@{}l@{}}
\left((2^2), (3, 1), (3, 1), 4 \right)\!, \\[0.8ex]
\left((2^2, 1^2), (3^2), (3^2), 6 \right)\!, \\[0.8ex]
\left((2^3), (3^2), (4, 1, 1), 6 \right)\!, \\[0.8ex]
\left((2^4), (3^2, 1^2), (4^2), 8 \right)\!, \\[0.8ex]
\left((2^5, 1^2), (3^4), (4^3), 12 \right)\!, \\[0.8ex]

\left((2^6), (3^4), (5^2, 1^2), 12 \right)\!, \\[0.8ex]
\left((2^{10}), (3^6, 1^2), (5^4), 20 \right)\!, \\[0.8ex]
\left((2^{14}, 1^2), (3^{10}), (5^6), 30 \right)\!, \\[0.8ex]
\left((2^{k}, 1), (2^{k}, 1), (2k+1), 2k + 1\right)\!, \quad k \ge 1, \\[0.8ex]
\left((2^{k}, 1^2), (2^{k+2}), (2k+2), 2k + 2 \right)\!, \quad k \ge 1.
\end{array}
\]
\end{theorem}

\begin{proof}
The theorem easily follows from Theorem~3 in~\cite{pak} and the classification of rational functions that are covering maps between orbifolds of positive Euler characteristic, which can be found in~\cite{pap1} and, in a slightly different form, in~\cite{gen}, where, in particular, the corresponding Belyi functions are explicitly computed. Below, we outline a more elementary approach.

Arguing as in the proof of Theorem~\ref{61}, we obtain that there exists a rational function $\phi$ such that the diagram 
\begin{equation} \label{h}
\begin{tikzcd}[row sep=2.7em, column sep=3.5em, baseline=(current bounding box.center)]
\mathbb{CP}^1 \arrow[r,"\phi"] \arrow[d,"z^{l/t}"']    
    & \mathbb{CP}^1 \arrow[d,"\beta_{k,m,l}"] \\
\mathbb{CP}^1 \arrow[r,"\beta"']    
    & \mathbb{CP}^1
\end{tikzcd}
\end{equation}
commutes. Furthermore, the chain rule implies that there exist $$z_1, z_2 \in \beta_{k,m,l}^{-1}(\{0,1,\infty\})$$ such that $\phi$ satisfies the conditions of Theorem~\ref{h6} with $k = l/t$. Choosing a representative of the equivalence class of clean Belyi functions $\beta_{k,m,l}$ in such a way that $z_1 = 0$ and $z_2 = \infty$, we may assume that $\phi$ has the form
\[
\phi = z^a R(z^{l/t}),
\]
where $R$ is a rational function and $a$ is not divisible by $l/t$.

Since the Belyi functions $\beta_{k,m,l}$ is invariant under the symmetry group of the corresponding Platonic solid, it easily follows that for any $t \mid l$, the representative of $\beta_{k,m,l}$ chosen above can be written in the form 
\begin{equation}\label{tho}
\beta_{k,m,l} = \hat{\beta}_{k,m,l}^t \circ z^{\frac{l}{t}}
\end{equation}
for some Belyi function $\hat{\beta}_{k,m,l}^t$, whose graph represents the quotient graph of the Platonic graph of type $(k,m,l)$ under the action of the cyclic group $\mathbb{Z}_{\frac{l}{t}}$ (the graphs in Fig.~\ref{f4} are examples of such quotient graphs).

The above implies that the diagram 
\begin{equation}\label{oc}
\begin{tikzcd}[row sep=2.7em, column sep=3.5em]
\mathbb{C}\mathbb{P}^{1} \arrow[r,"\phi"] \arrow[d,"z^{l/t}"'] 
    & \mathbb{C}\mathbb{P}^{1} \arrow[d,"z^{l/t}"] \\
\mathbb{C}\mathbb{P}^{1} \arrow[dr,"\beta"']
    & \mathbb{C}\mathbb{P}^{1} \arrow[d,"\hat{\beta}_{k,m,l}^t"] \\
    & \mathbb{C}\mathbb{P}^{1}
\end{tikzcd}
\end{equation}
commutes. Observing now that for any rational function $R$ and integers $a$ and $b$, the diagram 
\begin{equation} 
\begin{tikzcd}[row sep=2.7em, column sep=3.5em, baseline=(current bounding box.center)]
\mathbb{C}\mathbb{P}^1 \arrow[r,"z^aR(z^b)"] \arrow[d,"z^{b}"']    
    & \mathbb{C}\mathbb{P}^1 \arrow[d,"z^b"] \\
\mathbb{C}\mathbb{P}^1 \arrow[r,"z^aR^b(z)"']    
    & \mathbb{C}\mathbb{P}^1
\end{tikzcd}
\end{equation}
commutes, we see that diagram~\eqref{oc} can be completed to the commutative diagram 
\[
\begin{tikzcd}[row sep=2.7em, column sep=3.5em]
\mathbb{C}\mathbb{P}^{1} \arrow[r,"\phi"] \arrow[d,"z^{l/t}"'] 
    & \mathbb{C}\mathbb{P}^{1} \arrow[d,"z^{l/t}"] \\
\mathbb{C}\mathbb{P}^{1} \arrow[r,"\psi"'] \arrow[dr,"\beta"']
    & \mathbb{C}\mathbb{P}^{1} \arrow[d,"\hat{\beta}_{k,m,l}^t"] \\
    & \mathbb{C}\mathbb{P}^{1}
\end{tikzcd}
\]
where $\psi = z^a R^{l/t}(z)$, providing us with the decomposition 
\begin{equation}\label{eql} 
\beta = \hat{\beta}_{k,m,l}^t \circ \psi.
\end{equation} 
Further analysis of cyclic subgroups in $A_4$, $S_4$, $A_5$, and $D_{2n}$ leads to a description of the passports of $\hat{\beta}_{k,m,l}^t$, where, in fact, it is sufficient to consider only those with $t = 1$.
\end{proof}

Theorem~\ref{72} implies the following corollary in the same way that Theorem~\ref{71} yields Corollary~\ref{c71}.
f
\begin{corollary}\l{7c} 
The list of realizable passports of $2$-Platonic graphs that are non-uniform is as follows:
\[
\begin{array}{@{}l@{\qquad}l@{}}
\left((2^{2s}), (3^s,s), (3^s,s), 4s\right)\!,
& s \not\equiv 0 \pmod 3, \\[1ex]
\left((2^{2s},s^2), (3^{2s}), (3^{2s}), 6s\right)\!,
& s \not\equiv 0 \pmod 2, \\[1ex]
\left((2^{3s}), (3^{2s}), (4^s,s,s), 6s\right)\!,
& s \not\equiv 0 \pmod 4, \\[1ex]
\left((2^{4s}), (3^{2s},s^2), (4^{2s}), 8s\right)\!,
& s \not\equiv 0 \pmod 3, \\[1ex]
\left((2^{5s},s^2), (3^{4s}), (4^{3s}), 12s\right)\!,
& s \not\equiv 0 \pmod 2, \\[1ex]
\left((2^{6s}), (3^{4s}), (5^{2s},s^2), 12s\right)\!,
& s \not\equiv 0 \pmod 5, \\[1ex]
\left((2^{10s}), (3^{6s},s^2), (5^{4s}), 20s\right)\!,
& s \not\equiv 0 \pmod 3, \\[1ex]
\left((2^{14s},s^2), (3^{10s}), (5^{6s}), 30s\right)\!,
& s \not\equiv 0 \pmod 2, \\[1ex]
\left((2^{ks},s), (2^{ks},s), ((2k+1)^s), 2ks+s\right)\!,
& s \not\equiv 0 \pmod 2, \\[1ex]
\left((2^{ks},s^2), (2^{(k+1)s}), ((2k+2)^s), 2ks+2s\right)\!,
& s \not\equiv 0 \pmod 2,
\end{array}
\]
where $s \ge 2$ and $k \ge 1$ are integers.
\end{corollary}

\section{Further directions} 
\subsection{Uniform and 2-uniform graphs} 

The results of the previous section completely solve the realizability problem for passports of bicolored $2$-Platonic graphs. This naturally leads to a more general problem about the realizability of passports of  bicolored uniform and $2$-uniform graphs. 
More precisely, we can formulate the problem as follows. 

\pagebreak

\begin{problem} 
For a Platonic triple $(k,m,l)$, what are the realizable passports of the form 
\begin{equation} \label{izz0} 
\big((\overline{k}^*), (\overline{m}^*), (\overline{l}^*), n\big), 
\end{equation} 
or 
\begin{equation} \label{izz} 
\big((\overline{k}^*), (\overline{m}^*), (\overline{l}^*,u,v), n\big), \quad u,v\not\equiv 0 \pmod{l}, 
\end{equation} 
or 
\begin{equation} \label{izz2} 
\big((\overline{k}^*), (\overline{m}^*,u), (\overline{l}^*,v), n\big), \quad u\not\equiv 0 \pmod{m}, \ v\not\equiv 0 \pmod{l}?
\end{equation}
\end{problem}

The results of Section 7  provide handy necessary conditions for the existence of such passports. For example, for passports of the form~\eqref{izz0}, Theorem~\ref{71} implies that $n$ is divisible by $n(k,m,l)$. Moreover, the chain rule implies that the maximal degree of a black vertex does not exceed $\frac{n \cdot k}{n(k,m,l)}$. 

Similarly, for passports of the form~\eqref{izz} or \eqref{izz2}, Theorem~\ref{72} and Corollary~\ref{7c} provide necessary conditions for their realizability depending on their type in terms of the divisibility of $n$ and the maximum values of $u$ and $v$ or their sum. Nevertheless, it is not clear whether these conditions are \textit{sufficient}.

A possible way to approach this is to construct a rational function on the right-hand side of the equality $\beta = \hat{\beta} \circ \psi$ from Theorem~\ref{72} so that the left-hand side has the desired passport, or to prove that no such function can exist. Obviously, this is not as straightforward as in the case of $2$-Platonic graphs, where $\psi$ is a mere power function. However, it might turn out to be simpler than the original problem.

\subsection{Hurwitz existence problem for generalized Latt\`es maps} 
In a completely analogous way to $2$-uniform and $2$-Platonic graphs, one can define $3$-uniform and $3$-Platonic graphs, and more generally, bicolored $3$-uniform and $3$-Platonic graphs (note that the triple of parameters $(k,m,l)$ for such graphs is not necessarily Platonic). Transitioning to this class of objects obviously presents significant difficulties, and we are aware of only a few works on this topic, all concerning the $3$-Platonic case, namely, \cite{ak}, \cite{deza15}, \cite{fro3}, and \cite{ko}. 

We mention that Theorem~3 in \cite{pak} can also be applied to certain $3$-uniform bicolored graphs, providing decompositions in the spirit of Theorem~\ref{71}, where the left factors correspond now  to proper \textit{non-cyclic} subgroups of the automorphism group of the corresponding Platonic graph. Namely, it applies to passports of the following forms:
$$
\begin{array}{@{}ll@{}}
\big( (\overline{2}^*,u,v), (\overline{3}^*, w), (\overline{4}^*), n \big), & u,v \not\equiv 0 \pmod 2, \ w \not\equiv 0 \pmod 3, \\[0.8ex]
\big( (\overline{2}^*,u,v), (\overline{3}^*), (\overline{4}^*, w), n \big), & u,v \not\equiv 0 \pmod 2, \ w \equiv 2 \pmod 4, \\[0.8ex]
\big( (\overline{2}^*,u,v), (\overline{3}^*), (\overline{5}^*, w), n \big), & u,v \not\equiv 0 \pmod 2, \ w \not\equiv 0 \pmod 5, \\[0.8ex]
\big( (\overline{2}^*,u,v), (\overline{3}^*, w), (\overline{5}^*), n \big), & u,v \not\equiv 0 \pmod 2, \ w \not\equiv 0 \pmod 3, \\[0.8ex]
\big( (\overline{2}^*,u,v,w), (\overline{3}^*), (\overline{5}^*), n \big), & u,v,w \not\equiv 0 \pmod 2, \\[0.8ex]
\big( (\overline{2}^*,u), (\overline{3}^*, v,w), (\overline{5}^*), n \big), & u \not\equiv 0 \pmod 2, \ v,w \not\equiv 0 \pmod 3.
\end{array}
$$
\vskip 0.2cm

While referring the reader to the cited papers for further details, we want to pose the problem of describing certain $3$-uniform passports, which seems particularly interesting to us from two perspectives. First, this is the first large class of almost-uniform passports to which the methods of \cite{pak} do not provide a representation of the corresponding Belyi function in the form of a composition. Second, when these passports are realizable, the corresponding rational functions are \textit{generalized Lattès maps}, which naturally appear in many problems of complex and arithmetic dynamics. 

Directing the reader to \cite{lattes} (and to \cite{pak} in the context of the Hurwitz existence problem) for general definitions and properties of generalized Lattès maps, which, in general, have more than three critical values, we define the passports of interest explicitly as follows.

\begin{problem} \l{82}
What are the realizable passports of the form 
\[
\begin{array}{@{}ll@{}}
\big( (\overline{2}^*,u), (\overline{3}^*, v), (\overline{5}^*,w), n \big), & u \not\equiv 0 \pmod 2, \ v \not\equiv 0 \pmod 3, \ w \not\equiv 0 \pmod 5, \\[0.8ex]
\big( (\overline{2}^*,u), (\overline{3}^*, v), (\overline{4}^*,w), n \big), & u,w \not\equiv 0 \pmod 2, \ v \not\equiv 0 \pmod 3, \\[0.8ex]
\big( (\overline{2}^*,u), (\overline{3}^*, v), (\overline{3}^*,w), n \big), & u \not\equiv 0 \pmod 2, \ v,w \not\equiv 0 \pmod 3, \\[0.8ex]
\big( (\overline{2}^*,u), (\overline{3}^*, v, w), (\overline{3}^*), n \big), & u \not\equiv 0 \pmod 2, \ v,w \not\equiv 0 \pmod 3?
\end{array}
\]
\end{problem}

Although functions with such passports cannot be represented as compositions, they can be related to another highly remarkable class of objects: rational functions that commute with finite automorphism groups of $\mathbb{C}\mathbb{P}^1$. 

In more detail, let $f$ be a rational function whose passport, for example, is the first one from the list above, and assume that the points where the multiplicities of $f$ are exceptional are $0, 1$, and $\infty$. Consider the composition $f \circ \beta_{2,3,5}$. It is easy to see that the latter is a Belyi function with a uniform passport of type $(2,3,5)$. Consequently, by Theorem~\ref{71}, there exists a rational function $\phi$ such that the diagram 
\begin{equation}\label{d}
\begin{tikzcd}[row sep=2.7em, column sep=3.5em]
\mathbb{C}\mathbb{P}^1 \arrow[r,"\phi"] \arrow[d,"\beta_{2,3,5}"'] 
    & \mathbb{C}\mathbb{P}^1 \arrow[d,"\beta_{2,3,5}"] \\
\mathbb{C}\mathbb{P}^1 \arrow[r,"f"'] 
    & \mathbb{C}\mathbb{P}^1
\end{tikzcd}
\end{equation}
commutes.

Recall that $\beta_{2,3,5}$ is an invariant function for a group $G$ isomorphic to $A_5$, which is the symmetry group of the icosahedron, and it can be shown that the commutativity of diagram~\eqref{d} implies the existence of an automorphism $\sigma \colon G \to G$ such that
\begin{equation}\label{cond}
\phi \circ \mu = \sigma(\mu) \circ \phi, \qquad \mu \in G.
\end{equation}
Conversely, if a function $\phi$ satisfies~\eqref{cond} for some automorphism $\sigma$ of $G \cong A_5$, then it descends to the function $f$ that makes the diagram~\eqref{d} commute, and whose passport is the first one from the list above (see~\cite{lattes}). Similar results hold for other types of passports in the problem under consideration.

The above shows that the problem of describing functions $f$ with passports as mentioned in Problem~\ref{82} turns out to be equivalent to describing functions that satisfy~\eqref{cond} for groups $G$ isomorphic to $A_4$, $S_4$, or $A_5$. A particular case of~\eqref{cond} is when $\sigma = \mathrm{id}$, that is, when $\phi$ commutes with $G$. Moreover, one can easily see that while $\phi$ does not in general commute with $G$, some of its iterates does. On the other hand, functions commuting with finite subgroups of $\mathrm{Aut}(\mathbb{CP}^1)$ are described in~\cite{dm}. Although this description does not immediately yield a characterization of functions with the required passport, we believe this connection is quite remarkable and can be used to approach the Hurwitz existence problem for these passports.

\end{document}